\documentclass[11pt,reqno]{amsart}

\usepackage[T1]{fontenc}
\usepackage[utf8]{inputenc}
\usepackage[margin=1in]{geometry}
\usepackage{amsmath,amssymb,amsthm,amsfonts}
\usepackage{xcolor}
\usepackage{enumitem}
\usepackage{microtype}
\usepackage{indentfirst}
\usepackage[colorlinks=true,linkcolor={blue!75!black},citecolor={blue!80!black},urlcolor={blue!65!black}]{hyperref}

\numberwithin{equation}{section}

\theoremstyle{plain}
\newtheorem{theorem}{Theorem}[section]
\newtheorem{lemma}[theorem]{Lemma}
\newtheorem{corollary}[theorem]{Corollary}

\newcommand{\N}{\mathbb{N}}
\newcommand{\Nzero}{\mathbb{N}_0}
\newcommand{\HQ}{\mathbf{H}_Q}
\newcommand{\HNstar}{\mathbf{H}_{\N}^{\ast}}

\begin{document}

\title{Global Product Intersection Sets in Semigroups}

\author{Wouter van Doorn}
\address{Groningen, the Netherlands}
\email{wonterman1@hotmail.com}

\author{Pietro Monticone}
\address{Harmonic, London, United Kingdom}
\email{pietro.monticone@harmonic.fun}

\author{Quanyu Tang}
\address{School of Mathematics and Statistics, Xi'an Jiaotong University, Xi'an 710049, P. R. China}
\email{tang\_quanyu@163.com}


\begin{abstract}
For a family $(A_q)_{q\in Q}$ of subsets of a semigroup, the product
intersection set records those exponents $h \in \N$ for which the $h$-fold
product set of the intersection, $(\bigcap_q A_q)^h$, is equal to $\bigcap_q
A_q^h$, the intersection of the product sets. Nathanson recently asked which
subsets of $\N$ can occur as a product intersection set, both for arbitrary and
for decreasing families $(A_q)_{q\in Q}$. We solve both problems by giving a
complete classification. In particular, when $|Q| \ge 2$, we show that in either
case any subset $X \subseteq \N$ with $1 \in X$ occurs as a product intersection
set. Both classifications were autonomously discovered and formally verified in
Lean by \textnormal{\href{https://aristotle.harmonic.fun/}{Aristotle}}, a formal
reasoning agent developed by
\textnormal{\href{https://www.harmonic.fun/}{Harmonic}}.
\end{abstract}

\maketitle

\section{Introduction}

Let $S$ be a multiplicative semigroup and, with $Q$ an index set, let
$(A_q)_{q\in Q}$ be a family of subsets of $S$. We then write $A$ for the
intersection $\bigcap_{q\in Q} A_q$, where we interpret this intersection to be
the entire semigroup if $Q$ is the empty set. Further, for a subset $B\subseteq
S$, define the \emph{$h$-fold product set}
\[
B^h=\{b_1\cdots b_h:b_1,\dots,b_h\in B\},\qquad h\in\N,
\]
where throughout the paper, we write $\N:=\{1,2,3,\dots\}$. With these
definitions, one can quickly verify that $A^h \subseteq \bigcap_{q\in Q} A_q^h$
for every $h\in\N$~\cite[Eq.~(1)]{Nathanson2026Problems}. The set of exponents for which
this containment is an equality is called the
\emph{product intersection set}
\[
H_Q^S(A_q):=\left\{h\in\N: A^h=\bigcap_{q\in Q} A_q^h\right\}.
\]
For a fixed index set $Q$, the associated global set is
\[
\HQ:=\bigl\{H_Q^S(A_q): \text{$S$ is a semigroup and $(A_q)_{q\in Q}$
  is a $Q$-indexed family of subsets of $S$}\bigr\}.
\]
Thus, $\HQ$ is the collection of all subsets of $\N$ that arise as product
intersection sets for $Q$-indexed families in arbitrary semigroups.

In the special case $Q=\N$ we say that a sequence $(A_q)_{q\ge1}$ of subsets of
a semigroup $S$ is \emph{asymptotically strictly decreasing} if
$A_{q+1}\subseteq A_q$ for all $q\ge1$ and $A_q\ne A_{q+1}$ for infinitely many
values of $q$, i.e.\ the sequence is not eventually constant. The corresponding
global set is
\[
\HNstar:=\bigl\{H_{\N}^{S}(A_q): \text{$S$ is a semigroup and
  $(A_q)_{q\ge1}$ is asymptotically strictly decreasing}\bigr\}.
\]

The question of when the $h$-fold sumset of a decreasing intersection
agrees with the intersection of the $h$-fold sumsets was introduced
by Nathanson for additive abelian semigroups~\cite{Nathanson2026Intersections},
and was subsequently extended in~\cite{Nathanson2026Problems} to
product sets in arbitrary semigroups and to families indexed by general
sets $Q$, where the global sets $\HQ$ and $\HNstar$ were defined.
Nathanson~\cite[Theorem~14]{Nathanson2026Problems} proved that these two global sets
are closed under intersections and asked various natural follow-up
questions~\cite[Problems~10 and~11, and the second part of Problems~4
and~7]{Nathanson2026Problems}:

\begin{itemize}[leftmargin=2em]
\item Are $\HQ$ and $\HNstar$ also closed under unions?
\item For a fixed index set $Q$, does there exist a set $X\subseteq\N$ with
  $1\in X$ and such that either $X\notin\HQ$ or $X\notin\HNstar$?
\end{itemize}

The purpose of this paper is to answer both questions, by exactly describing
$\HQ$ and $\HNstar$. More precisely, in Section~\ref{sec:HNstar} we prove the
following result for $\HNstar$.

\begin{theorem}\label{thm:hnstar-classification}
We have \[
\HNstar=\{X\subseteq\N: 1\in X\}.
\]
\end{theorem}

As for $\HQ$, in Section~\ref{sec:HQ} we show that it is fully determined by the
cardinality of $Q$, in the following way.

\begin{theorem}\label{thm:hq-classification}
We have
\[
\HQ=
\begin{cases}
\{\{1\},\N\}, & Q=\varnothing, \\
\{\N\}, & |Q|=1,\\
\{X\subseteq\N:1\in X\}, & |Q|\ge 2.
\end{cases}
\]
\end{theorem}

We note that answers to Nathanson's two questions immediately follow from
Theorems~\ref{thm:hnstar-classification} and~\ref{thm:hq-classification}.
Indeed, for both sets the answer to the first question is positive, while the
answer to the second question is negative, unless $|Q|\le 1$ in the case of
$\HQ$. We further remark that for non-empty $Q$ the same results hold true if we
restrict to monoids instead of semigroups. To see this, let $S$ be a semigroup
and let $(A_q)_{q \in Q}$ be a family of subsets of $S$, where $Q \neq
\varnothing$. Now, if $S$ is a subsemigroup of $T$, the sets $A^h$ and $A_q^h$
stay the same, regardless of whether they are computed in $S$ or in $T$. We
therefore have the equality
\[
H_{Q}^{S}(A_q) = H_{Q}^{T}(A_q).
\]
In particular, if $T$ is the monoid obtained from $S$ by adjoining an identity
element, we find that any product intersection set realized by a semigroup can
be realized by a monoid as well.

For a brief overview of the proofs, assume for the moment $|Q| \ge 2$ (as
the proof for $|Q| \le 1$ will be fairly straightforward), and let $n \ge 2$
be an integer. Then, in order to prove either Theorem~\ref{thm:hnstar-classification}
or~\ref{thm:hq-classification}, we construct an
explicit semigroup $S_n$ and an explicit family of subsets $(A_{n,q})_{q\in Q}$
of $S_n$ (which is decreasing in the case of $\HNstar$) such that
\begin{equation} \label{eq:single-exclusion}
H_Q^{S_n}(A_{n,q}) = \N\setminus\{n\}.
\end{equation}

We then combine these single-exclusion building blocks via product
constructions to produce arbitrary subsets of $\N$ containing $1$. Before we do
this, however, we first need a few preliminary lemmas.

\section{Elementary product lemmas}
When $h = 1$, both sides of the containment $A^h \subseteq \bigcap_{q} A_q^h$
simplify to $\bigcap_{q} A_q$, so that, in particular, the sets are equal. We
therefore obtain the following necessary condition for sets in $\HQ$ and
$\HNstar$.

\begin{lemma}\label{lem:onealways}
For every set $Q$ and for every $Q$-indexed
family $(A_q)_{q\in Q}$ of subsets of a semigroup $S$, we have $1 \in H_Q^S(A_q)$.
Hence, if $X \in \HQ$ or $X \in \HNstar$, then $1 \in X$.
\end{lemma}

We now record two basic lemmas for products of semigroups.

\begin{lemma}\label{lem:boxproduct}
Let $(S_i)_{i\in I}$ be a family of semigroups. For subsets $B_i,C_i\subseteq S_i$,
consider the box products
\[
B=\prod_{i\in I} B_i,\qquad C=\prod_{i\in I} C_i
\]
as subsets of $\prod_{i\in I} S_i$. Then
\[
BC=\prod_{i\in I} (B_iC_i).
\]
\end{lemma}

\begin{proof}
If $x=bc\in BC$ with $b=(b_i)\in B$ and $c=(c_i)\in C$, then $x_i=b_ic_i\in
B_iC_i$ for every $i$, hence $x\in\prod_i(B_iC_i)$. This proves
$BC\subseteq\prod_i(B_iC_i)$.

Conversely, let $x=(x_i)\in\prod_i(B_iC_i)$. For each $i$ choose $b_i\in B_i$
and $c_i\in C_i$ with $x_i=b_ic_i$. Then $b=(b_i)\in B$, $c=(c_i)\in C$, and
$x=bc\in BC$. Hence $\prod_i(B_iC_i)\subseteq BC$.
\end{proof}

\begin{corollary}\label{cor:boxpower}
With the notation of Lemma~\ref{lem:boxproduct}, for every $h\in\N$,
\[
\left(\prod_{i\in I} B_i\right)^h=\prod_{i\in I} B_i^h.
\]
\end{corollary}

\begin{proof}
This follows by induction on $h$. The case $h=1$ is trivial. If the statement
holds for $h$, then by Lemma~\ref{lem:boxproduct}
\[
\left(\prod_i B_i\right)^{h+1}
=\left(\prod_i B_i\right)^h\left(\prod_i B_i\right)
=\left(\prod_i B_i^h\right)\left(\prod_i B_i\right)
=\prod_i (B_i^hB_i)
=\prod_i B_i^{h+1}. \qedhere
\]
\end{proof}

\begin{lemma}\label{lem:boxinter} Let $I$ and $Q$ be index sets and let
$(A_{i,q})_{i\in I,q\in Q}$ be a doubly indexed family of subsets of
semigroups $(S_i)_{i\in I}$. Then
\[
\bigcap_{q\in Q}\left(\prod_{i\in I} A_{i,q}\right)=\prod_{i\in I}\left(\bigcap_{q\in Q}A_{i,q}\right).
\]
\end{lemma}

\begin{proof}
An element $a=(a_i)$ belongs to the left-hand side if and only if, for every
$q\in Q$ and every $i\in I$, one has $a_i\in A_{i,q}$. This is equivalent to
saying that, for every $i\in I$, one has $a_i\in\bigcap_{q\in Q}A_{i,q}$, which
is exactly the condition for $a$ to belong to the product on the right-hand
side.
\end{proof}

\section{The set \texorpdfstring{$\HNstar$}{HN*}}
\label{sec:HNstar}

The goal of this section is to prove Theorem~\ref{thm:hnstar-classification}, by first
constructing a semigroup $S_n$ and a decreasing sequence of subsets
$(A_{n,q})_{q\ge1}$ of $S_n$ such that \eqref{eq:single-exclusion} holds. For
this, we define $\mathcal{W}$ to be the set of all non-empty words over the
alphabet $\N$, and let $\alpha$ and $\beta$ be two distinct symbols not
contained in $\mathcal W$.

\subsection{A single-exclusion semigroup}

Fix an integer $n \ge 2$, and let $\mathcal W_{<n}$ be the subset of
$\mathcal{W}$ that consists of all words of length $<n$. We then define $S_n :=
\mathcal W_{<n}\cup\{\alpha,\beta\}$, and we will turn $S_n$ into a semigroup.

For words $u,v\in\mathcal W_{<n}$, let $uv$ denote concatenation. We define a
multiplication on $S_n$ as follows:
\begin{equation*}
\alpha \cdot x = x \cdot \alpha = \alpha,\qquad \beta \cdot x = x \cdot \beta = \alpha\qquad (x\in S_n),
\end{equation*}
and for $u,v\in\mathcal W_{<n}$,
\[
u\cdot v=
\begin{cases}
uv, & |u|+|v|<n,\\
\beta, & |u|+|v|=n,\\
\alpha, & |u|+|v|>n.
\end{cases}
\]

\begin{lemma}\label{lem:Tnsemigroup}
$S_n$ is a semigroup.
\end{lemma}

\begin{proof}
As $S_n$ is certainly closed under the multiplication operation we defined, it
is sufficient to prove that this multiplication is associative. Now, if at least
one of $x,y,z$ is $\alpha$ or $\beta$, then both $(x \cdot y) \cdot z$ and $x
\cdot (y \cdot z)$ are equal to $\alpha$, because any product with one factor
equal to $\alpha$ or $\beta$ equals $\alpha$. It therefore further suffices to
consider the case in which $x,y,z$ are elements of $\mathcal W_{<n}$, say $x=u$,
$y=v$, $z=w$. With $L := |u|+|v|+|w|$, there are then three cases to consider.

If $L<n$, then both $u \cdot v = uv$ and $v \cdot w = vw$ are words of length
$<n$, and
\[
(u \cdot v) \cdot w = u \cdot (v \cdot w) = uvw,
\]
which is again a word of length $L<n$.

If $L=n$, then $(|u|+|v|)+|w| = |u|+(|v|+|w|) = n$, so that
\[
(u \cdot v) \cdot w = u \cdot (v \cdot w) = \beta.
\]

If $L>n$, we claim that both $(u \cdot v) \cdot w$ and $u \cdot (v \cdot w)$ are
equal to $\alpha$. To prove $(u \cdot v) \cdot w = \alpha$, first assume
$|u|+|v|\ge n$. In that case $u \cdot v \in \{\alpha,\beta\}$, and therefore $(u
\cdot v) \cdot w = \alpha$. On the other hand, if $|u|+|v|<n$, then $u \cdot v =
uv$ with $|uv| + |w| = L > n$, so that $(u \cdot v) \cdot w = \alpha$. As the
symmetric argument applies to $u \cdot (v \cdot w)$, we do indeed get
\[
(u \cdot v) \cdot w = u \cdot (v \cdot w) = \alpha.
\]
We conclude that the multiplication we defined on $S_n$ is associative, and
hence we see that $S_n$ is a semigroup.
\end{proof}

For each $q\ge 1$, define
\[
A_{n,q} := \{\alpha\} \cup \{[m]: m\ge q\} \subseteq S_n,
\]
where $[m]$ denotes the one-letter word whose only letter is $m$.

\begin{lemma}\label{lem:singledecrease}
The sequence $(A_{n,q})_{q\ge 1}$ is strictly decreasing and
$\bigcap_{q\ge 1} A_{n,q}=\{\alpha\}$.
\end{lemma}

\begin{proof}
Clearly $A_{n,q+1}\subsetneq A_{n,q}$ because $[q] \in A_{n,q} \setminus
A_{n,q+1}$. Equally clearly, $\alpha \in A_{n,q}$ for all $q$, while $[m] \notin
A_{n,m+1}$, so that $\bigcap_{q\ge 1} A_{n,q}=\{\alpha\}$.
\end{proof}

\begin{lemma}\label{lem:powersEnq}
Let $n \ge 2$ and $h \in \N$.
\begin{enumerate}[label=\textnormal{(\roman*)},leftmargin=2em]
\item If $1 \le h < n$, then $A_{n,q}^h = \{\alpha\} \cup \{[m_1 \cdots m_h]: m_1, \ldots, m_h \ge q\}$.
\item If $h=n$, then $A_{n,q}^h = \{\alpha, \beta\}$.
\item If $h>n$, then $A_{n,q}^h = \{\alpha\}$.
\end{enumerate}
\end{lemma}

\begin{proof}
For $h<n$, any product of $h$ elements of $A_{n,q}$ is either $\alpha$ (if one factor is
$\alpha$) or the concatenation of $h$ one-letter words, hence a word of length $h$ all of
whose letters are at least $q$. This proves the inclusion ``$\subseteq$''.
Conversely, every word $[m_1\cdots m_h]$ with all $m_i\ge q$ is the product
$[m_1]\cdots[m_h]$, so the reverse inclusion holds as well.

For $h=n$, the same argument shows that any product of $n$ factors different
from $\alpha$ must be the product of $n$ one-letter words, and by definition
such a product is equal to $\beta$. Thus $A_{n,q}^n\subseteq\{\alpha, \beta\}$.
The reverse inclusion is clear: $\alpha \in A_{n,q}^n$, and $\beta =[q]^n\in
A_{n,q}^n$.

If $h>n$, any product containing a factor $\alpha$ is equal to $\alpha$. On the
other hand, if all $h$ factors are one-letter words, then the total length is
$h>n$, so the product is $\alpha$ by the definition of multiplication. Hence
$A_{n,q}^h=\{\alpha\}$.
\end{proof}

\begin{theorem}\label{thm:HN-single-exclusion}
For every $n\ge 2$,
\[
H_{\N}^{S_n}(A_{n,q})=\N\setminus\{n\}.
\]
\end{theorem}

\begin{proof}
By Lemma~\ref{lem:singledecrease}, the intersection of the sequence is
$A_n:=\{\alpha\}$, and therefore $A_n^h=\{\alpha\}$ for every $h\in\N$.

If $1\le h<n$, then by Lemma~\ref{lem:powersEnq}(i), an element of $A_{n,q}^h
\setminus \{\alpha\}$ is a word of length $h$ whose letters are all at least
$q$; no such word can belong to the intersection $\bigcap_{q\ge 1}A_{n,q}^h$
over all $q$. Hence
\[
\bigcap_{q\ge 1}A_{n,q}^h=\{\alpha\}=A_n^h.
\]

If $h=n$, Lemma~\ref{lem:powersEnq}(ii) gives
\[
\bigcap_{q\ge 1}A_{n,q}^h = \{\alpha, \beta\} \ne \{\alpha\} = A_n^h.
\]
Hence $n\notin H_{\N}^{S_n}(A_{n,q})$.

Finally, if $h>n$, Lemma~\ref{lem:powersEnq}(iii) gives
\[
\bigcap_{q\ge 1}A_{n,q}^h=\{\alpha\}=A_n^h.
\]
Thus $H_{\N}^{S_n}(A_{n,q})=\N\setminus\{n\}$.
\end{proof}

\subsection{Classification of \texorpdfstring{$\HNstar$}{HN*}}
We can now prove our classification theorem on $\HNstar$.

\begin{proof}[Proof of Theorem~\ref{thm:hnstar-classification}]
By Lemma~\ref{lem:onealways}, every member of $\HNstar$ contains $1$, so it
suffices to prove the converse. Hence, with $X$ a set of positive integers with
$1\in X$, we aim to construct a semigroup $S$ and a strictly decreasing sequence
$(A_q)$ such that $H_{\N}^{S}(A_q) = X$.

If $X = \N$, then we can take $S$ to be the multiplicative semigroup
$(\Nzero,\cdot)$, with
\[
A_q := \{0\} \cup \{m \in \Nzero: m \ge q\} \qquad (q \ge 1).
\]

Then $(A_q)_{q\ge 1}$ is strictly decreasing, as $q \in A_q\setminus A_{q+1}$.
Furthermore, for all $h, q \in \N$ we have both $q \notin A_{q+1}$ and
$q \notin A_{q+1}^h$. Hence, for every $h$ we have
\[
\bigcap_{q\ge 1} A_q^h=\{0\}=\{0\}^h=\left(\bigcap_{q\ge 1}A_q\right)^h,
\]
implying $H_{\N}^{\Nzero}(A_q) = \N$, as desired.

Assume now that $X\ne\N$, and let $I := \N\setminus X \subseteq\{2,3,4,\dots\}$.
For every $n\in I$, let $S_n$ and $(A_{n,q})_{q\ge 1}$ be the building blocks
constructed above. We then define the product semigroup
\[
S:=\prod_{n\in I} S_n
\]
and the sequence of subsets
\[
A_q:=\prod_{n\in I} A_{n,q}\subseteq S\qquad (q\ge 1),
\]
which is strictly decreasing because any sequence $(A_{n,q})$ with fixed $n \in
I$ is.

Now, by Lemmas~\ref{lem:boxinter} and~\ref{lem:singledecrease} we get
\[
A=\prod_{n\in I}\left(\bigcap_{q\ge 1} A_{n,q}\right)=\prod_{n\in I}\{\alpha\},
\]
so that, by Corollary~\ref{cor:boxpower},
\[
A^h=\prod_{n\in I}\{\alpha\}^h=\prod_{n\in I}\{\alpha\},
\]
for all $h \in \N$. On the other hand,
\[
\bigcap_{q\ge 1} A_q^h
=\bigcap_{q\ge 1}\left(\prod_{n\in I} A_{n,q}^h\right)
=\prod_{n\in I}\left(\bigcap_{q\ge 1}A_{n,q}^h\right).
\]
This latter product is equal to $A^h$ if, and only if, all intersections
$\bigcap_{q\ge 1}A_{n,q}^h$ are equal to $\{\alpha\}$. By the proof of
Theorem~\ref{thm:HN-single-exclusion}, this happens precisely when $h \notin I$.
That is, $h\in H_{\N}^{S}(A_q)$ exactly when $h \in X$, finishing the proof.
\end{proof}

\section{The set \texorpdfstring{$\HQ$}{HQ}}
\label{sec:HQ}

This section will be devoted to determining $\HQ$.

\begin{proof}[Proof of Theorem~\ref{thm:hq-classification}]
Starting off with the easiest case, assume that $|Q| = 1$, say $Q = \{q\}$. Then
$A = A_q$, so that in particular $A^h = A_q^h$ for all $h \in \N$. We
therefore get that the product intersection set is equal to $\N$,
regardless of the semigroup $S$ and the sequence $(A_q)_{q\in Q}$.

Moving on to even fewer elements, let $Q$ be empty and recall that for the empty
family of subsets of a semigroup $S$, the intersection is $S$ itself. Hence the
associated product intersection set is $\{h\in\N:S^h=S\}$. We claim that this
set is always either $\{1\}$ or $\N$.

To see this, if $S^2=S$, then by induction $S^h=S$ for all $h\in\N$: once
$S^m=S$, we have
\[
S^{m+1}=S^mS=SS=S^2=S.
\]
Conversely, if $S^2\ne S$, then for
every $h\ge 2$ one has $S^h\subseteq S^2\ne S$, so $S^h\ne S$. Thus the only
possibilities are $\{1\}$ and $\N$, and both possibilities do actually occur.
Indeed, any monoid satisfies $S^h=S$ for all $h$, realizing $\N$. On the other
hand, let $S=\{0,1\}$ with multiplication $xy=0$ for all $x,y\in S$. Then
$S^h=\{0\}$ for every $h\ge 2$, so this semigroup realizes $\{1\}$.

Finally, we consider the case $|Q| \ge 2$, where we let $q_1$ and $q_2$ be two
distinguished elements of $Q$. In analogy with the proof of
Theorem~\ref{thm:hnstar-classification}, for each $n \ge 2$ we then construct a
semigroup together with a family of subsets, for which the product intersection
set is exactly $\N\setminus\{n\}$. So fix an integer $n \ge 2$, and equip the
set $S_n := \{0, 1, \ldots, n^3 + n^2\}$ with the operation
\[
x \star y = \min\{x+y, n^3 + n^2\}.
\]
Then $(S_n, \star)$ is a semigroup (a commutative monoid
with identity $0$, in fact), and by straightforward induction we have the
following formula for the $h$-fold product:

\begin{lemma}\label{lem:trunc-sum}
For all $h \in \N$ and all $a_1,\dots,a_h\in S_n$, we have
\[
a_1\star\cdots\star a_h=\min\{a_1+\cdots+a_h,n^3 + n^2\}.
\]
In particular, $a^h=\min\{ha,n^3 + n^2\}$ for all $a\in S_n$.
\end{lemma}

We now define the subsets $B_n=\{n^2,n^2+1\}$ and $C_n=\{n^2,n^2+n\}$ of $S_n$,
for which we have the following identities for the corresponding powers:

\begin{lemma}\label{lem:Pn-powers}
For all $h \le n$ we have
\[
B_n^h=\{hn^2+j: 0\le j\le h\}
\qquad\text{and}\qquad
C_n^h=\{hn^2+jn: 0\le j\le h\},
\]
while $B_n^h=C_n^h=\{n^3+n^2\}$ for all $h>n$.
\end{lemma}

\begin{proof}
If $h \le n$, then any element of $B_n^h$ is the sum of $j$ copies of $n^2+1$
and $h-j$ copies of $n^2$, for some $j$ with $0 \le j \le h$, which indeed gives
\[
B_n^h=\{hn^2+j: 0\le j\le h\}.
\]
On the other hand, if $h > n$, then any sum of $h$ elements of $B_n$ is at least
$(n+1)n^2 = n^3 + n^2$, so that $B_n^h$ consists of the single element $n^3 +
n^2$. The analogous statements for $C_n^h$ are proved in the same way.
\end{proof}

With the above, the product intersection set is not too hard to determine.

\begin{theorem}\label{thm:pair-single-exclusion}
For every $n\ge 2$,
\[
\{h\in\N:(B_n\cap C_n)^h=B_n^h\cap C_n^h\}=\N\setminus\{n\}.
\]
\end{theorem}

\begin{proof}
As $B_n\cap C_n=\{n^2\}$, we have
\[
(B_n\cap C_n)^h=\{\min\{hn^2,n^3+n^2\}\}
\]
for every $h\in\N$, by Lemma~\ref{lem:trunc-sum}. By
Lemma~\ref{lem:Pn-powers}, this agrees with $B_n^h\cap C_n^h$ for $h>n$, but not
for $h=n$, where $n^3 + n$ is also contained in $B_n^h \cap C_n^h$. As for $h <
n$, we note that any element $hn^2 + jn$ of $C_n^h$ with $j \ge 1$ is larger
than any element $hn^2 + j' \in B_n^h$. Hence, for $h < n$ we also have
\[
(B_n\cap C_n)^h=B_n^h\cap C_n^h= \{hn^2\}. \qedhere
\]
\end{proof}

We can now finish the proof of Theorem~\ref{thm:hq-classification}. Let $X$ be
any set of positive integers with $1\in X$. For $X=\N$, we can take any monoid
$S$ and set $A_q=S$ for all $q\in Q$, which gives $H_Q^S(A_q)=\N$. So assume
that $X\ne\N$, and let $I := \N \setminus X \subseteq\{2,3,4,\dots\}$. For each
$n\in I$, let $S_n$, $B_n$, and $C_n$ be as above, and form the product monoid
\[
S :=\prod_{n\in I} S_n.
\]
Further, define subsets of $S$ by
\[
B :=\prod_{n\in I} B_n,
\qquad
C :=\prod_{n\in I} C_n,
\]
and let $(A_q)_{q\in Q}$ be defined by
\[
A_q :=
\begin{cases}
B, & q=q_1,\\
C, & q=q_2,\\
S, & q\in Q\setminus\{q_1,q_2\}.
\end{cases}
\]
Then $\bigcap_{q\in Q}A_q=B\cap C$. Moreover, since $S$ is a monoid, $S^h=S$ for
every $h\in\N$, and therefore $\bigcap_{q\in Q} A_q^h=B^h\cap C^h$. Thus
$H_Q^S(A_q)=\{h\in\N:(B\cap C)^h=B^h\cap C^h\}$.

By Lemma~\ref{lem:boxinter} and Corollary~\ref{cor:boxpower},
\[
(B\cap C)^h
=\left(\prod_{n\in I}(B_n\cap C_n)\right)^h
=\prod_{n\in I}(B_n\cap C_n)^h,
\]
while
\[
B^h\cap C^h
=\left(\prod_{n\in I}B_n^h\right)\cap\left(\prod_{n\in I}C_n^h\right)
=\prod_{n\in I}(B_n^h\cap C_n^h).
\]
Hence
\[
h\in H_Q^S(A_q)
\iff \forall n\in I,\; (B_n\cap C_n)^h=B_n^h\cap C_n^h.
\]
By Theorem~\ref{thm:pair-single-exclusion}, the latter is equivalent to
$h\ne n$ for every $n\in I$, that is, to $h\in X$.
Therefore $H_Q^S(A_q)=X$, and so $X\in\HQ$ as desired.
\end{proof}

\newpage

\section*{Appendix: Formal Discovery and Verification with Aristotle}
\label{sec:appendix}

The proofs in this paper were autonomously discovered and formally verified in
Lean by \emph{Aristotle}, a formal reasoning agent
developed by \textnormal{\href{https://www.harmonic.fun/}{Harmonic}}~\cite{Achim2025} and
publicly available for free at
\href{https://aristotle.harmonic.fun/}{\nolinkurl{aristotle.harmonic.fun}}.

We asked Aristotle to characterise $\HQ$ and $\HNstar$ separately,
then to combine both characterisations into a complete solution to
Problems~10 and~11 in~\cite{Nathanson2026Problems}. During the characterisation of
$\HNstar$, our prompt did not specify the meaning of ``asymptotically
strictly decreasing,'' and Aristotle adopted a definition stronger
than the one in~\cite[Section~2]{Nathanson2026Problems}. A simple argument shows
that the two definitions yield the same set $\HNstar$, and the Lean code
verifies this equivalence explicitly. We then asked Aristotle to
examine whether there were any other problems in \cite{Nathanson2026Problems} that
could also be resolved, and it observed that
\cite[Problems~4 and~7]{Nathanson2026Problems} are in fact consequences of the
characterisations of $\HQ$ and $\HNstar$ that it had obtained at the outset.

The full Lean formalisation can be inspected interactively in the
\href{https://live.lean-lang.org/#project=mathlib-v4.28.0&url=https://gist.githubusercontent.com/pitmonticone/a55e4bcf3f036e302725bf5f1616c135/raw/d3171921134d1a1def48f322fb00034e40074274/Nathanson_IPSSS_P10_P11.lean}%
{Lean~4 Web Editor}. Aristotle also informalised the Lean file into
a first draft of the present paper, which we subsequently rewrote
and polished.

\end{document}